\documentclass[12pt]{amsart}

\usepackage{include}

\usepackage[margin=2.9cm]{geometry}

\title{On the Unique Ergodicity of Quadratic Differentials and the Orientation Double Cover}
\date{\today}

\begin{document}

\author{M. E. Smith}

\address{Department of Mathematics, University of Utah,  Salt Lake City, UT, USA} 

%\address{Department of Mathematics, University of Utah, 155 S 1400 E Room 233, Salt Lake City, UT 84112} 

\email{msmith@math.utah.edu}

\maketitle
%\hypertarget{TOC}{\tableofcontents}

\begin{abstract}
We construct an example of a uniquely ergodic measured foliation on a surface such that the associated translation flow on the orientation double cover is minimal but not uniquely ergodic. We then prove a geometric criterion for the horizontal foliation of a quadratic differential to be uniquely ergodic. The second theorem generalizes a result of Trevi\aTL{n}o for the horizontal flow on a translation surface \cite{MR3177386}, as well as Masur's criterion for unique ergodicity of the horizontal foliation \cite{MR1167101}.
\end{abstract}

\section{Introduction}

A \E{flat surface} is a pair $(S,\Ph)$ where $S$ is a compact Riemann surface and $\Ph$ is a meromorphic quadratic differential on $S$ with at most simple poles. The differential $\Ph$ determines a flat Riemannian metric with conical singularities at the set $\Si$ of zeroes and poles of $\Ph$. We are most interested in the case when $(S,\Ph)$ has \E{finite area} in terms of the area form of $\Ph$, which forces any poles of $\Ph$ to be simple. If $\Ph$ is holomorphic the $\Ph$-area will always be finite. In this case $\Ph$ determines a Borel probability measure $\la$ on $S$, which we call \E{Lebesgue measure}. We would like to know if there are any other Borel probability measures naturally associated with $(S,\Ph)$.

Away from the zeroes or poles, we can write $\Ph$ in local coordinates as $\ph(z) (\Rd z)^2$ for some holomorphic function $\ph$. The kernel of $\Im\b(\sqrt{\ph(z)} \, \Rd z\e)$ determines a singular measured foliation of $S$, which we call the \E{horizontal foliation}. Following Masur and Smillie \cite{MR1135877}, we say that a Borel probability measure $\mu$ on $S$ is \E{ergodic} for $\Ph$ if for every $\mu$-measurable set $E \se S$ which is a union of leaves of the horizontal foliation, $\mu(E) = 0$ or $\mu(E)  = 1$. A quadratic differential $\Ph$ is \E{uniquely ergodic} if Lebesgue measure is the only ergodic measure.

If $\al$ is an abelian differential on $S$, then $(S,\al)$ is called a \E{translation surface}: away from zeroes of $\al$, there exist charts for $S$ whose transition maps are all translations. Since $\al^2$ is a quadratic differential, we have a singular Riemannian metric and a Lebesgue measure on $(S,\al)$. Further, the kernel of $\Im \al$ determines a parallel unit vector field $X$ on $S \m \Si$. In this case the horizontal foliation of $\al^2$ is given by a unit speed flow along $X$, called the \E{horizontal flow}. This flow always preserves the Lebesgue measure. If $\mu$ is an ergodic probability measure for the horizontal flow, then by observing that for any measurable $E$ which is a union of leaves and flow invariant, $\mu(E) = 0$ or $\mu(E) = 1$, justifying our definition of ergodicity for quadratic differentials.

For any flat surface which is not a translation surface we can construct a unique branched cover $p : \HS \to S$ which "unfolds" $\Ph$ to an abelian differential $\al$. Statements about transverse measures associated to the horizontal foliation of $\Ph$ can be restated in terms of the horizontal flow of $\al$. Our first result is negative: passing to the orientation double cover can increase the number of ergodic measures.

\begin{thm}
There exists a compact flat surface $(S,\Ph)$ whose horizontal foliation is uniquely ergodic, but on the orientation double cover $(\HS,\al)$, the horizontal flow is minimal but not uniquely ergodic.
\label{th:double NUE}
\end{thm}

Despite being a universal object, the orientation double cover can fail to capture all of the dynamical properties of the original foliation.

Our second result is positive, and gives a criterion for the horizontal foliation of a compact flat surface to be uniquely ergodic. The criterion is based on the natural $\SL_2(\R)$ action on the moduli space of  quadratic differentials, given by post-composing every chart in a flat atlas with the usual action of $\SL_2(\R)$ on $\R^2$. In particular, consider the subgroup of diagonal matrices:
\[
g_t = \begin{bmatrix} e^{-t} & 0 \\ 0 & e^{t} \end{bmatrix}
\]
The action of $g_t$ on $\Ph$ determines a flow on the moduli space of flat surfaces called the \E{Teichm\"uller flow}. There are many connections between the Teichm\"uller flow orbit of $\Ph$ and the ergodic measures associated to $\Ph$. In our case, we will consider the \E{systole} of a flat surface, which is the length of the shortest simple closed curve on $(S,\Ph)$ which is not homotopically trivial. Our main theorem says that if the systole does not shrink rapidly along the Teichm\"uller orbit of the surface, then the horizontal foliation is uniquely ergodic.
\begin{thm}
Let $\ka(t)$ be the systole of the compact flat surface $g_t \dt (S, \Ph)$. If:
\begin{equation}
\int_0^{\oo} \ka(t)^2 \, \Rd t = \oo
\label{eq:systole}
\end{equation}
then the horizontal foliation determined by $\Ph$ is uniquely ergodic.
\label{th:foliation unique ergodicity}
\end{thm}

This theorem is based on work of Rodrigo Trevi\aTL{n}o, who proved the same geometric criterion for unique ergodicity of the horizontal flow of an abelian differential $\al$ \cite[Theorem 4]{MR3177386}.

\begin{thm}[Trevi\aTL{n}o]
Let $\ka(t)$ be the systole of the compact translation surface $g_t \dt (S, \al)$. If:
\begin{equation}
\int_0^{\oo} \ka(t)^2 \, \Rd t = \oo
\end{equation}
then the horizontal flow determined by $\al$ is uniquely ergodic.
\label{th:rodrigo theorem 4}
\end{thm}

Theorem \ref{th:foliation unique ergodicity} implies Masur's criterion for unique ergodicity, which is the case where $\limsup_{t \to \oo} \ka(t) > 0$.

\begin{thm}[Masur's criterion \cite{MR1167101}] If the horizontal foliation of a flat surface is minimal and not uniquely ergodic, then the Teichm\"uller orbit (of the class of that flat metric) leaves every compact set of the moduli space.
\label{th:masur criterion}
\end{thm}

Theorem \ref{th:foliation unique ergodicity} also implies a result of Cheung and Eskin. Let $\De_t$ denote the length of the smallest saddle connection on $g_t \dt (S, \Ph)$, and let $d(t) = -\log \De_t$. By an argument of Trevi\aTL{n}o \cite{MR3177386}, we can choose $\ep = \frac{1}{2}$ in the following theorem.
\begin{thm}[Cheung and Eskin \cite{MR2388697}] 
There is an $\ep > 0$ such that if $d(t) <  \ep \log t + C$ for some C and for all $t > 0$, then the horizontal foliation is uniquely ergodic.
\end{thm}

By the example in Theorem \ref{th:double NUE}, Theorem \ref{th:foliation unique ergodicity} cannot be a direct consqeuence of Trevi\aTL{n}o's theorem in general. If the quadratic differential $\Ph$ is not orientable but has no odd order zeroes or poles, then $p : \HS \to S$ is a genuine connected covering space. In this case there is a proof of unique ergodicity given Trevi\aTL{n}o's criterion which is much shorter than our proof in the general case. The proof is a consequence of the following theorem.

\begin{thm}
Suppose that $(S,\Ph)$ is a flat surface and $p : R \to S$ is a finite degree genuine covering map (i.e. without branching). If the systoles of $g_t \dt (S,\Ph) $ satisfy the systole condition (\ref{eq:systole}), then so do the systoles of $g_t \dt (R, p^*\Ph)$. In particular, the horizontal foliation (or flow) on $(R,p^*\Ph)$ is uniquely ergodic by Theorem \ref{th:foliation unique ergodicity}.
\label{th:cover ergodicity}
\end{thm}

When the orientation cover is a genuine cover, Theorem \ref{th:foliation unique ergodicity} is a direct consequence of Theorem \ref{th:rodrigo theorem 4} and Theorem \ref{th:cover ergodicity}. Theorem \ref{th:cover ergodicity} has the following consequence for covers of typical flat surfaces.

\begin{cor}
For almost every flat surface $(S,\Ph)$ (with respect to the Masur-Veech measure on strata), the horizontal foliations of both $(S,\Ph)$ and every finite degree genuine cover $(R,p^*\Ph)$ of $(S,\Ph)$ are simultaneously uniquely ergodic. In particular, almost every flat surface which is a finite degree genuine cover is uniquely ergodic.
\label{cor:generic cover}
\end{cor}
This follows from Theorem \ref{th:cover ergodicity} and the ergodicity of Teichm\"uller geodesic flow with respect to the Masur-Veech measure, which implies almost every flat surface is recurrent. Masur and Veech independently constructed the measure and proved ergodicity in \cite{MR644018} and \cite{MR644019}. Note that surfaces which are genuine covers of lower genus surfaces form a measure zero set with respect to the Masur-Veech measure.

\B{Outline of the paper:} First, we need to clarify the definition of the systole and the transverse measure in order for Trevi\aTL{n}o's methods to apply to our situation. This is done in Section \ref{sec:defs}.

The example in Theorem \ref{th:double NUE} is constructed in Section \ref{sec:double NUE}. We consider a family of minimal translation flows which are minimal but not uniquely ergodic, then show that they are all orientation double covers of uniquely ergodic foliations.

Our argument for Theorem \ref{th:foliation unique ergodicity} proceeds in two steps: We first prove Theorem \ref{th:foliation ergodicity}, based on \cite[Theorem 2]{MR3177386}, which implies that Lebesgue measure is ergodic for the horizontal foliation if there exists a nice "thick-thin" decomposition of the surface along its Teichm\"uller flow orbit. Such a decomposition exists by using the systole condition (\ref{eq:systole}). %Our more general theorem also guarantees ergodicity of the foliation in the non-compact case, which is of independent interest. 
Our proof heavily uses technology developed by Forni \cite{MR1888794}. Note that we cannot just move the geometry to the double cover and cite Trevi\aTL{n}o, by the example in Theorem \ref{th:double NUE}. Indeed, in many steps we will need to use the covering map to translate computations back and forth between the original surface and its orientation double cover.

In Section \ref{sec:foliation unique ergodicity} we finish the proof of Theorem \ref{th:foliation unique ergodicity}. Using the systole condition again, we can upgrade from ergodicity to unique ergodicity. The major difficulty is in constructing a map in the style of Veech \cite[Section 1]{MR516048} which normalizes a particular measure to the Lebesgue measure without distorting the geometry too much. We need to modify the known map in the abelian differential case so that it will work for quadratic differentials.

In Section \ref{sec:covers}, we provide a proof of Theorem \ref{th:cover ergodicity} using topological techniques for covering maps that are not available in the general case. As the counterexample in Theorem \ref{th:double NUE} demonstrates, one needs to be careful about statements relating the measures of an orientation double cover to the original flat surface. We do need $p : \HS \to S$ to be a covering map for this proof to work, and so Theorem \ref{th:foliation unique ergodicity} is not just a simple corollary of Trevi\aTL{n}o's theorem except in special cases. We do not expect Theorem \ref{th:cover ergodicity} to hold for covers when Trevi\aTL{n}o's criterion does not apply.

\B{Remarks:} There is a gap in Trevi\aTL{n}o's proof of the ergodicity criterion, as no argument was given for why \cite[Equation 27]{MR3177386} should hold. We present an argument to fix this gap in Lemma \ref{lem:essential gradient bound}, which is essential to proving ergodicity for Theorem \ref{th:foliation unique ergodicity}. Trevi\aTL{n}o's paper also considered flows on non-compact translation surfaces with finite area, while our paper only considers compact flat surfaces. We were unable to find a proof of Lemma \ref{lem:essential gradient bound} which also works for translation flows on non-compact translation surfaces, and our proof uses compactness of the surface several times. We believe that it may be possible to find a proof which does not use compactness, in which case we would obtain Trevi\aTL{n}o's ergodicity criterion \cite[Theorem 2]{MR3177386} for non-compact flat surfaces as well, however, such a proof would require knowledge of the behavior of singularities on non-compact surfaces. See the technical remark at the end of the proof of Theorem \ref{th:foliation ergodicity} and equation (\ref{eq:infinite integral noncompact}) for details on what we can currently prove for the non-compact case.

 The example of Theorem \ref{th:double NUE} is a meromorphic quadratic differential on a surface of genus $g = 1$. On higher genus surfaces, it may be possible to find orientation double covers with other interesting effects. In particular, the orientation cover may double some ergodic measures and not others. Given any $k \ge 1$ and $k \le n \le 2k$, is it possible to find a quadratic differential whose horizontal foliation is minimal and has $k$ ergodic measures, and such that the horizontal flow on the orientation double cover is minimal and has $n$ ergodic measures? We believe the answer is yes. We also believe that it is possible to find examples where the quadratic differential is holomorphic.

%Our proof of the Lemma uses compactness, and is the only place in the proof of Theorem \ref{th:foliation ergodicity} where compactness is used in an essential way. We were unable to find an argument that gives this bound in the non-compact case. A different proof of the same bound that works in the non-compact case would give us a similar ergodicity condition for diverging non-compact surfaces.

\B{Acknowledgements:} The author was partially supported by NSF grants DMS-135500 and DMS-1452762, as well as the research training grant DMS-1246989. I would like to thank Leonard Carapezza for helpful conversations. I am also grateful to Rodrigo Trevi\aTL{n}o and an anoynmous referee for providing many helpful suggestions and corrections on initial drafts of the paper. Finally, I would like to thank Jon Chaika for suggesting this problem and all of his advice during the creation of this paper, and Mladen Bestvina, who asked me to find the counterexample in Theorem \ref{th:double NUE}.

\section{Definitions}
\label{sec:defs}

Given a flat surface $(S,\Ph)$ we can construct its \E{orientation double (branched) cover}: there is a translation surface $(\HS, \al)$ and a degree 2 branched covering map  $p : \HS \to S$ so that $p^* \Ph = \al^2$. The branch points of $p$ are the odd order zeroes or poles of $\Ph$, and away from the branch points, $p$ is a local isometry. $\HS$ is disconnected if and only if $\Ph = \al^2$ for some abelian differential $\al$. 
%See \cite[Figure 1]{MR2439000} for an example. (remove citation because I have my own example now.
The translation surface $\HS$ comes with an involution, an isometry $\io : \HS \to \HS$ which exchanges $\al$ and $-\al$ and fixes the ramification points of $p$. We can think of $(S,\Ph)$ as the orbifold quotient of $(\HS, \al)$ by the involution $\io$. The standard construction is described in more detail in Lemma \ref{lem:universal property}. Figure \ref{fig:minimal NUE} depicts the orientation double cover of the surface in Figure \ref{fig:double NUE} (both figures are in Section \ref{sec:double NUE}).

\begin{defn} We say that a Borel probability measure $\mu$ on a compact flat surface $(S,\Ph)$ is an \E{invariant measure for the horizontal foliation} if there is an invariant Borel probability measure $\nu$ for the horizontal flow on $(\HS,\al)$ such that $\mu = p_*\nu$. There is a simplex $C$ of invariant measures for the horizontal flow, and we say that an invariant measure is \E{ergodic} if it is an extreme point of the simplex $p_*C$.
\label{def:QD ergodic measure}
\end{defn}

The purpose of this definition is to ensure some kind of "flow invariance" for measures on $S$, even though we do not have a flow on $S$. We could also consider all \E{transverse invariant measures}. These gadgets give lengths to arcs transverse to the horizontal foliation and are invariant under homotopy along the foliation.% or \E{holonomy}. 

\begin{prop}
Every projection of a horizontal flow invariant measure on $\HS$ determines a transverse invariant measure for the horizontal foliation on $S$, and vice versa.
\label{prop:QD ergodic measure}
\end{prop}
\begin{proof}
Each projection $p_*\nu$ of an invariant measure $\nu$ for the horizontal flow of $\al$ determines a transverse invariant measure $\Upsilon$ for the horizontal foliation of $\Ph$:
\begin{equation}
\Upsilon(\ga) = \lim_{\ep \to 0} \frac{1}{\ep} p_* \nu ([0,\ep] \dt \ga)
\end{equation}
where $[0,\ep] \dt \ga$ is the set formed by flowing $\ga$ along the foliation for all times between $0$ and $\ep$. As $\nu$ is flow-invariant upstairs, $\Upsilon$ will be a transverse invariant measure.

Conversely, given a transverse invariant measure $\Upsilon$ for the horizontal foliation of $\Ph$, we can construct an invariant measure $\mu$ on $S$. Indeed, if $\ga$ is a small enough curve contained in a leaf of the vertical foliation of $\Ph$ in a neighborhood of a regular point, then we can flow $\ga$ for all times in $[0,\ep]$ for some small $\ep > 0$ to create a rectangle $R$. We define $\mu(R) = \Upsilon(\ga) \times \ep$ for such rectangles $R$, which uniquely determines a finite Borel measure $\mu$ on $S$. Viewing $\HS$ measurably as two copies of $S$, the measure $\nu = \frac{1}{2} \b(\mu + \io_* \mu \e)$ is a finite measure on $\HS$ with $p_* \nu = \mu$. Because $\Upsilon$ has a transverse invariance property for homotopy along the foliation, $\nu$ is an invariant measure for the horizontal flow on $\HS$.

\end{proof}

Our definition of ergodicity reflects dynamical properties of the foliation, as the next proposition shows.
\begin{prop}
The following are equivalent for an invariant measure $\mu$ for the horizontal foliation of a compact flat surface $(S,\Ph)$ of finite area.
\begin{enumerate}
\item $\mu$ is ergodic for the horizontal foliation.
\item For any $f \in L^2(\mu)$ such that $f$ is constant along leaves, $f$ is constant $\mu$-almost everywhere.
\item For any $f \in L^2(\mu)$ such that $p^*f = f \o p$ is invariant under the horizontal flow, $f$ is constant $\mu$-almost everywhere.
\item For every measurable $E \se S$ which is a union of leaves, $\mu(E) = 0$ or $\mu(E) = 1$.
\end{enumerate}
\label{prop:ergodic definitions}
\end{prop}
The proof is standard.

%\begin{proof} %\B{1 if and only if 2.} Similar to the usual proof for ergodic measures of dynamical systems. If $f \in L^2(\mu)$ satisfies $p^* f$ invariant but is not constant $\mu$-a.e., %then we can decompose $\mu$ as a convex combination of two invariant probability measures, so $\mu$ is not extreme. If $\mu$ is not extreme, then $\mu = s \nu_1 + (1-s)\nu_2$ for %%$(0 < s < 1)$ and distinct measures $\nu_i$.  We can take the Radon-Nikodym derivatives $f_i$ of $\nu_i$ with respect to $\mu_i$. Then $p^* f_i$ give flow-invariant functions which are %not constant almost everywhere.
%\B{2 if and only if 3.} Again, this is similar to the usual proof. If $f \in L^2(\mu)$ is real valued and not constant almost everywhere and $p^* f$ is flow-invariant, then for some $a \in \R$ %the sets $E = \{f\1((-\oo,a])\}$ and $E^c$ both have positive $\mu$-measure. These sets are unions of leaves: $E = p(p\1(E))$ is the projection of a flow invariant set. Similarly, if there %is a set $E$ which is a union of leaves and $0 < \mu(E) < 1$, then the function $p^* \chi_E$ is a flow-invariant function which is not constant almost everywhere. Since $\mu$ is a probability %measure, $\chi_E \in L^2(\mu)$. 
%\end{proof}

%The crucial part of this proposition is our ability to detect ergodicity for the horizontal foliation using only functions on $S$ and the covering map.

We also need to carefully define the systole. The Riemann sphere $\C\P^1$ can support meromorphic quadratic differentials with finite area, however, there are no homotopically nontrivial curves on a sphere, and so there is not immediately a natural shortest length. In the literature, poles of order 1 for a meromorphic quadratic differential are thought of as punctures on the surface $S$. Taking these punctures into account gives the correct definition.

\begin{defn}
Let $(S,\Ph)$ be a compact flat surface of finite area, and let $P \se S$ be the set of poles of $\Ph$. Let $S^* = S \m P$. The \E{systole} $\ka$ of $(S,\Ph)$ is the flat length of the shortest essential simple closed curve on $S^*$. That is, $\ka$ is the shortest length among curves $\ga$ which are homotopically nontrivial on $S^*$ and which are not homotopic to any puncture.
\label{def:systole}
\end{defn}

There are curves homotopic to punctures with arbitrarily small length, so we exclude them in order to make the systole positive. Note that $\C\P^1$ must be punctured at least four times in order to support a meromorphic quadratic differential of finite area, see \cite[Chapter IV, Section 12]{MR743423}. With at least four punctures, there are curves on the punctured sphere which are homotopically nontrivial and not homotopic to any puncture, so the systole is well defined for surfaces with enough complexity to support flat structures of finite area. By Mumford's compactness criteria, the only way for a sequence of these flat surfaces to escape to infinity is for each surface to have a short essential simple closed curve, shrinking to length zero. 
%This is true in the flat world by an extremal length calculation but I'm not sure who did it first.

\section{A Nonergodic Double Cover}
\label{sec:double NUE}

\begin{proof}[Proof of Theorem \ref{th:double NUE}]
First we construct a translation surface $(\HS,\al)$ by gluing two identical tori along a slit in such a way that the horizontal flow is minimal but not uniquely ergodic. For example, one could begin with a the standard torus of area 1 given by a unit square in the plane. We can skew the square using the matrix:
\[
\begin{bmatrix}
1 & 0\\
R & 1\\
\end{bmatrix}
\]
The first return map of the horizontal flow on the skewed torus to a vertical cross-section is rotation by $R$. By gluing two such tori along a vertical segment of length $K$, we obtain a translation surface of genus 2 as in Figure \ref{fig:minimal NUE}, with $\al$ given by $\Rd z$ in the plane. By a theorem of Veech, for any irrational $R$ with unbounded partial quotients in its continued fraction expansion, there are uncountably many $K$ such that the horizontal flow is minimal but not uniquely ergodic \cite{MR0240056}. See also \cite{MR0207961} and \cite{MR0391184} for related constructions. Let $\mu_0,\mu_1$ denote the two ergodic measures, both of which must be absolutely continuous with respect to Lebesgue measure.

\begin{figure}[h]	% The [h] means "insert here." To really force the figure to be where you want it use [H]. The indenting shown below is only for readability.
	\centering
	\includegraphics[scale = 0.4]{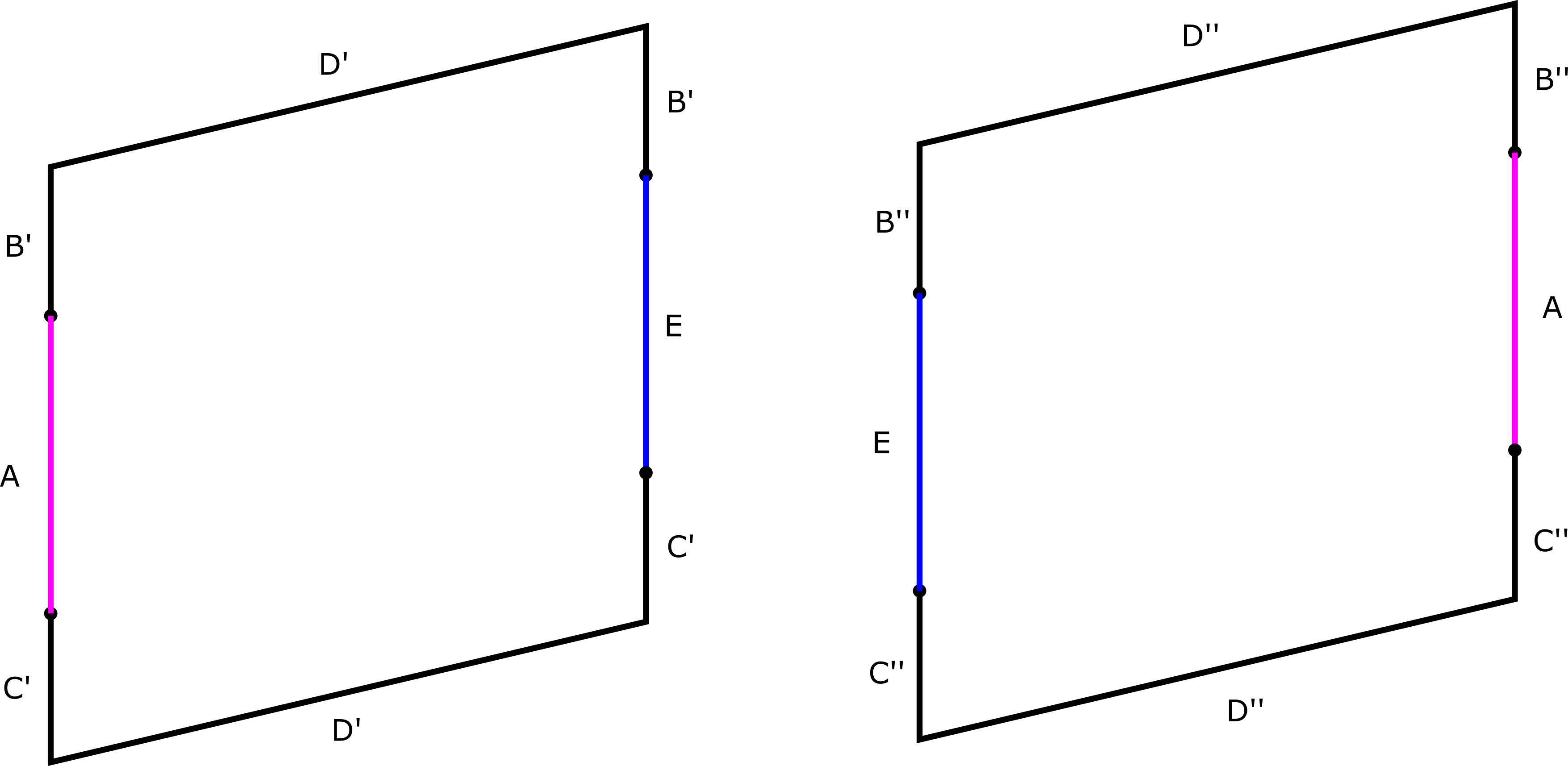}	% You can scale (here, by 70%). You could also set the width by using [width=3.0in].																			%   If you don't use any [] argument, the figure will be inserted at its original size. ex [scale = 0.8]
	\caption{A polygon representation of the translation surface $(\HS,\al)$. Sides with the same label are glued to each other by translation. The slit is given by the blue and purple cuves labeled A and E. The two zeroes of $\al$ each have order 1 and lie at the endpoints of A, which are the same as the endpoints of E.}
	\label{fig:minimal NUE}
\end{figure}

There is a holomorphic involution $\si : \HS \to \HS$ given by swapping the labels of the two tori. If we view the polygons for our surface as lying in $\C \x \Z/2$ as in Figure \ref{fig:minimal NUE}, then $\si$ is given by the map $\si(z,j) = (z,j + 1 \md 2)$, and $\si^*(\al) = \al$. Note that $\si_*(\mu_0) = \mu_1$ and vice versa. There is another holomorphic involution given by rotating both tori by 180 degrees, which is the hyperelliptic involution $h : \HS \to \HS$. In terms of the polygons, $h(z,j) = (-z,j)$, and so $h^*(\al) = - \al$. Observe that $h$ and $\si$ commute, and so the map $\io = h \o \si$ is an order two conformal automorphism of $\HS$ such that $\io^* \al = - \al$. Define $S$ to be $\HS / \g{\io}$. Because $\io^*(\al^2) = \al^2$ but $\io^* (\al) = -\al$, $\al^2$ descends to a quadratic differential $\Ph$ on $S$ which is not the square of an abelian differential. The flat surface $(S,\Ph)$ is an element of the strata $\CQ(2,-1,-1)$ of quadratic differentials, as depicted in Figure \ref{fig:double NUE}, and can be thought of as a "torus with a pocket". We would like to show that $(\HS, \al)$ is the orientation double cover of $(S,\Ph)$, and that the horizontal foliation of $(\HS,\Ph)$ is uniquely ergodic.

\begin{figure}[h]	% The [h] means "insert here." To really force the figure to be where you want it use [H]. The indenting shown below is only for readability.
	\centering
	\includegraphics[scale = 0.4]{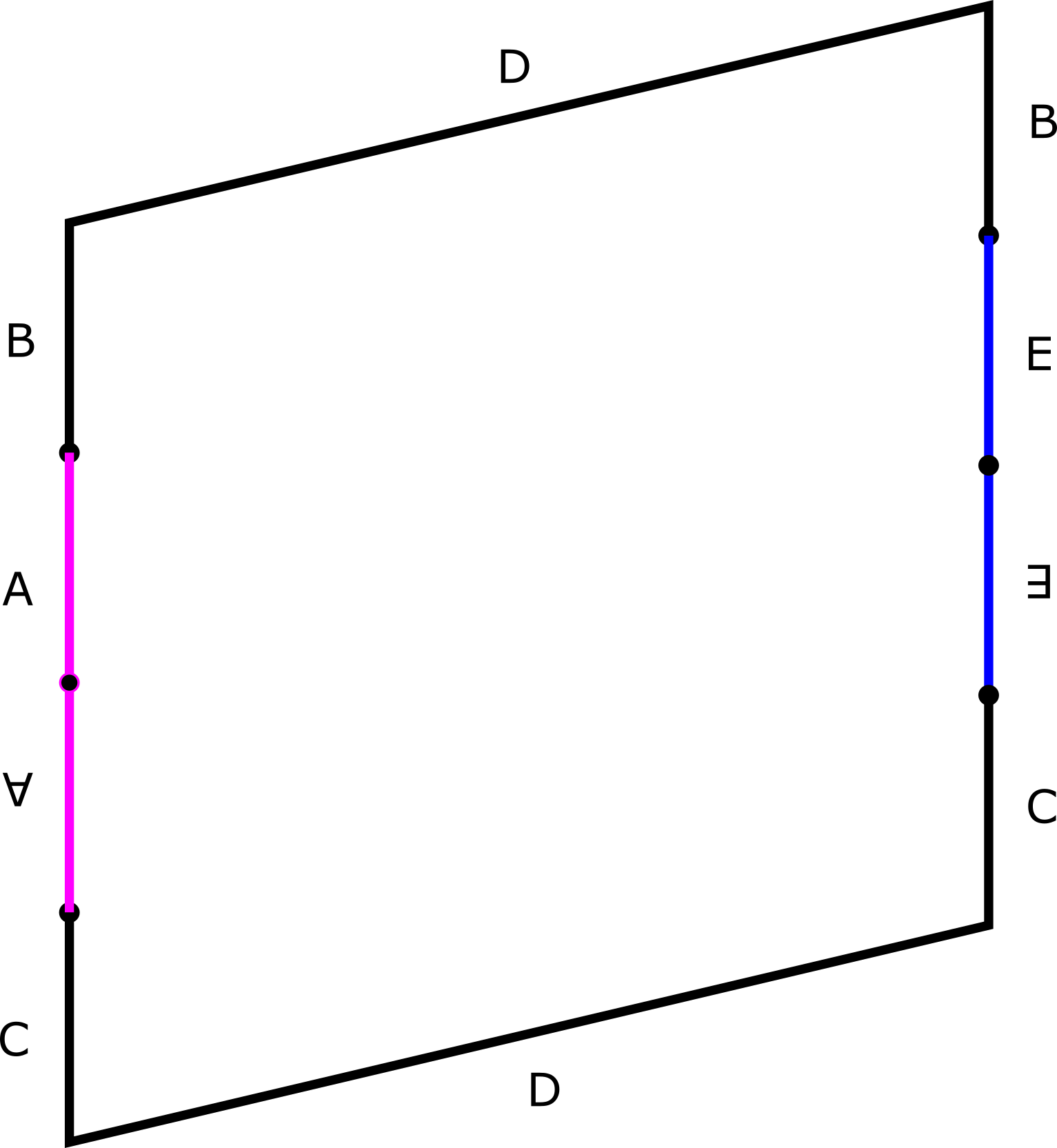}	% You can scale (here, by 70%). You could also set the width by using [width=3.0in].																			%   If you don't use any [] argument, the figure will be inserted at its original size. ex [scale = 0.8]
	\caption{A polygon representation of the flat surface $(S,\Ph)$. Sides with the same label are glued to each other by translation, except for the sides labeled A and E, which are glued by a rotation by $\pi$. The two simple poles of $\Ph$ are where A and E are glued to their rotated copies, and the single zero of order 2 lies at the other endpoint of A, which is identified to the other endpoint of E.}
	\label{fig:double NUE}
\end{figure}

To show $(\HS, \al)$ is the orientation double cover of $(S,\Ph)$, we will use the following universal property.

\begin{lem}[Universal Property of the Orientation Double Cover]
Let $(X,\Psi)$ be a compact flat surface where $\Psi$ is not the square of an abelian differential, and $(\HX,\om)$ its orientation double cover with projection map $\pi : \HX \to X$, $\pi^* (\Psi) = \om^2$. If $(Y,\be)$ is any translation surface with a holomoprhic map $f : Y \to X$ such that $f^*(\Psi) = \be^2$, then there is a unique holomorphic map $\Hf : Y \to \HX$ such that $f = \pi \o \Hf$ and $\Hf^*(\om) = \be$.
\label{lem:universal property}
\end{lem}

\begin{proof}
First we recall the standard construction of the orientation double cover. Let $Z \se X$ be the set of zeroes and poles of $\Psi$. Choose an atlas $\{U_i\}$ of charts for $X \m Z$ where each chart is connected and simply connected, so that $\sqrt{\Psi}$ is well defined on each $U_i$. The orientation double cover $(\HX, \om)$ is constructed by taking two copies $U_i^{\pm1}$ of each chart $U_i$, corresponding to each local choice of $\pm \sqrt{\Psi}$ on $U_i$. We glue two sets $U_i^a$ and $U_j^b$ together by the following rule: when $U_i \i U_j$ is nonempty, we identify the part of $U_i^a$ in the intersection with the part of $U_j^b$ in the intersection whenever $a\sqrt{\Psi} = b\sqrt{\Psi}$. This defines a punctured Riemann surface $\TX$, and a holomorphic map $\pi : \TX \to X$ defined by $\pi(x,\pm1) = x$. On $\TX$, there is an abelian differential $\om$ defined locally by $a\sqrt{\Psi}$ on each $U_i^a$. The surface $\TX$ can be extended to the surface $\HX$ by filling in the punctures in a canonical way, and $\om$ extends to $\HX$ as well.

Given the map $f : Y \to X$ with $f^*(\Psi) = \be^2$, away from $Z$ and the branch/ramification loci for $f$ we can define $\Hf : Y \to \HX$ in local coordinates  by $\Hf(y) = (f(y),s)$. Here $s \in \{+1,-1\}$ is the sign of $\pm\sqrt{\Psi}$ so that in any of the sets $U_i$ as above containing $f(y)$, $f^*(s\sqrt{\Psi}) = \be$. This map is well defined. Indeed, if $U_j$ is another open set containing $f(y)$, then there is some sign $s'$ such that $f^*(s'\sqrt{\Psi}) = \be$ on $U_j$. This can only happen if $s\sqrt{\Psi} = s'\sqrt{\Psi}$ on the part of $U_i \i U_j$ containing $f(y)$, which is exactly the gluing condition for the orientation double cover.

The map $\Hf$ is then holomorphic, extends to the zeroes of $\be$ by Riemann's extension theorem, and satisfies $\pi \o \Hf = f$ and $\Hf^* (\om) = \be$ by definition of $\pi$ and $\om$.

\end{proof}

Now let $(\HX,\om)$ be the orientation double cover of $(S,\Ph)$ with covering map $\pi$. By the universal property, the quotient map $q : \HS \to S$ factors through a map $\Hq : \HS \to \HX$ with $\Hq^*(\om) = \al$. As both covering maps $\pi$ and $q$ are degree 2, the map $\Hq$ is a degree 1 map of compact Riemann surfaces, since $q = \pi \o \Hq$. Hence $\Hq$ is a biholomorphism preserving the flat structure, so $(\HS,\al)$ can be canonically identified as the orientation double cover of $(S,\Ph)$.

To show that the horizontal foliation of $(\HS,\Ph)$ is uniquely ergodic, it suffices to show that the map $\io$ switches the two ergodic measures. Indeed, any invariant probability measure for the horizontal flow on $(\HS,\al)$ is given by $t \mu_0 + (1 - t)\mu_1$ for some $t \in [0,1]$ So, any invariant measure $\nu$ for the horizontal foliation of $(S,\Ph)$ must have the form $q_*(t \mu_0 + (1 - t)\mu_1) = tq_*(\mu_0) + (1-t)q_*(\mu_1)$. Observe that since $\io$ switches $\mu_0$ and $\mu_1$, and $q : \HS \to S$ is given by modding out by the involution $\io$, $q_*(\mu_0) = q_*(\mu_1)$: For any Borel set $E$ on $S$,
\[
q_*(\mu_0)(E) = \mu_0(q^{-1}E) = \mu_0(\io^{-1} q^{-1} E) = \io_* (\mu_0)(q^{-1}(E)) = \mu_1(q^{-1}E) = q_*(\mu_1)(E).
\]
Hence the measure $\nu$ is just $q_*(\mu_0) = q_*(\mu_1)$. Since this is the unique invariant measure, it must be the natural Lebesgue measure given by $\Ph$.

Since the map $\si$ switches the two ergodic measures on $(\HS,\al)$, it suffices to show that $h$ leaves the two measures invariant. Then $\io = h \o \si$ will switch the two ergodic measures. For $\mu$ an ergodic measure for the horizontal flow, let $A_\mu \in H_1 (\HS;\R)$ be the asymptotic cycle associated to $\mu$, as in \cite{MR0088720} and \cite{MR0331438}. The asymptotic cycle is nonzero and uniquely determines the ergodic measure. So, if we can determine how $h$ acts on asymptotic cycles, then we determine how it changes the ergodic measures for the horizontal flow. Observe the following properties:
\begin{enumerate}
\item The hyperelliptic involution $h$ satisfies $h^*(\al) = - \al$, and so $h$ is a topological conjugacy between the horizontal flow and its inverse flow.
\item The invariant measures for a flow and the inverse flow are the same.
\item If $A_\mu$ is an asymptotic cycle for the horizontal flow, then $-A_\mu$ is the asymptotic cycle for the inverse flow for the same measure $\mu$.
\item The map $h_* : H_1(\HS;\R) \to H_1(\HS;\R)$ is given by $h_*(a) = -a$.
\item Since $h_*(A_\mu) = -A_\mu$, it follows that $h_*(\mu) = \mu$ for every ergodic measure for the horizontal flow.
\end{enumerate}
So the map $\io$ exchanges the two ergodic measures for the horizontal flow, and so we have constructed a flat surface $(S,\Ph)$ whose horizontal foliation is uniquely ergodic, but the horizontal flow on the orientation double cover $(\HS,\al)$ is minimal but not uniquely ergodic.

\end{proof}

\section{An Ergodicity Criterion}
\label{sec:foliation ergodicity}

In this section we prove the following criterion for ergodicity of Lebesgue measure for a horizontal foliation. In the next section we will prove that the criterion also implies unique ergodicity.

\begin{thm}
Let $\ka(t)$ be the systole of the compact flat surface $S_t = g_t \dt (S,\Ph)$. If:
\begin{equation}
\int_0^{\oo} \ka(t)^2 \, \Rd t = \oo
\label{eq:systole2}
\end{equation}
then the horizontal foliation determined by $\Ph$ is ergodic.
\label{th:foliation ergodicity}
\end{thm}

%The original ergodicity condition for the horizontal flow in \cite[Theorem 2]{MR3177386} was designed for non-compact surfaces and needs to allow for $\ep(t) \to 0$ as $t \to \oo$. In the context of Theorem \ref{th:foliation unique ergodicity}, $\ep(t)$ can be chosen to be constant along the Teichm\"uller orbit (see the discussion by equation \ref{eq:systole diameter}).
%The integral condition implying that divergent leaves have measure zero in the translation surface case is based off of an improvement to this theorem in \cite[Theorem 19]{2015arXiv150300389H}.

We would like to call attention to the different roles of four small parameters that appear in the proof: $\et, \ep, \de_t, \tx{ and } \ka(t)$. The number $\et$ will be an arbitrary positive number. Given $\et$, we find $\ep$ by removing an open set of total measure at most $\et$ containing a ball of radius $\ep$ in $S_t$ centered at every singular point $\si \in \Si$. We will call the complement of this open set $S_{\ep,t}$. After removing these open sets, the resulting surface $S_{\ep,t}$ is likely disconnected. There will be $C_t$ components, call them $S_t^i$, and the $i$-th component has diameter $\CD^i_t$. The parameter $\de_t$ is the \E{injectivity radius} of $S_t$, and also measures how close we need to travel to a singularity if we want to travel in $S_t$ between two connected components of $S_{\ep,t}$. In general $\de_t$ is much smaller than $\ep$, and reflects properties of the "thin" part of the surface, especially for divergent Teichm\"uller geodesics. The systole $\ka(t)$ is equal to $2\de_t$. Indeed, $\de_t$ is the flat injectivity radius of the surface $S_t$, which is half of the length of the shortest essential simple closed curve.

\begin{proof}[Proof of Theorem \ref{th:foliation ergodicity}]
%First we prove that the integral condition implies that the leaf through Lebesgue almost every point is defined for all time. Suppose for a contradiction that the set of points $x$ where the leaf through $x$ terminates after a finite distance in either direction has positive measure. Then there is a real number $s_0$ and a set $B \se S$ of bad points such that $\la(B) > 0$ and the leaf through each $b \in B$ is not defined after a distance $s_0$ in one direction along the leaf... and follow Rodrigo/pat Hooper

Our proof of ergodicity follows the strategy for \cite[Theorem 2]{MR3177386}, with the additional complication of the orientation covering map. Fix $(S,\Ph)$ as in the hypothesis and choose an abelian differential $\al$ on $\HS$ with $p^* \Ph = \al^2$. Recall that when $(S_t,\Ph_t)$ is the Teichm\"uller geodesic flow of the surface by $g_t$, the Lebesgue measure on $S$ is preserved by $g_t$. Hence  $L^2(S_t) = L^2(S)$ for all $t$. Now we fix a function $u \in L^2(S)$ such that $p^* u$ is invariant under the translation flow. We assume that $u$ is real valued and has zero average, and want to show that $u \ee 0$ Lebesgue almost everywhere.

The computation $u \ee 0$ proceeds by a series of lemmas. First, we will decompose $u = m_t + h_t$, where $m_t$ is a meromorphic function on $S$. We will then show that $m_t$ is zero, which is the longest step and where the flat geometry plays the largest role. Finally, we will show that $h_t$ is zero using other properties of the geometry and dynamics of the foliation. 

%This last step is where the additional assumptions in the noncompact case need to be used.

\begin{lem}
The function $u$ can be written as $u = m_t + h_t$ in $L^2(S_t)$, where $m_t$ is meromorphic on $S_t$ and $h_t$ is orthogonal to $m_t$. Further, $p^*u = p^* m_t + \d_t v_t$ for a differential operator $\d_t$ determined by $\al_t$.
\end{lem}

\begin{proof}
Let $\CM_t$ be the space of meromorphic functions in $L^2(S_t)$. By using the Cauchy integral formula, one can show if a sequence of meromorphic functions converges in $L^2(S_t)$, then the sequence converges uniformly on compact sets of $S_t \m \Si$. Hence the $L^2(S_t)$ limit of a sequence of meromorphic functions must be meromorphic, and so $\CM_t$ is a closed supspace of $L^2(S_t)$. So there is an orthogonal decomposition $L^2(S_t) = \CM_t \op \CM_t^{\perp}$, and so $u = m_t + h_t$ for unique $m_t \in \CM_t$ and $h_t \in \CM_t^{\perp}$.

%NOTE: Uniform convergence on compact sets is by Stein and Shakarchi's Complex Analysis, Chapter 3, Exercise 20

Now observe that the linear map $p^* : L^2(S_t) \to L^2(\HS_t)$ defined by $p^*f = f \o p$ isometrically embeds $L^2(S_t)$ as a closed subspace of $L^2(\HS_t)$. Indeed,
\[
\int_{\HS_t} |p^* f|^2 \, \Rd \la = 2 \times \frac{1}{2} \int_{S_t} |f|^2 \, \Rd (p_* \la)
\]
as $p$ is a degree 2 covering map almost everywhere, we can view $\HS_t$ measurably as two copies of $S_t$ with half the measure. Hence $\|p^* f\|_{L^2(\HS_t)} = \|f\|_{L^2(S_t)}$. This shows that $p^*$ is an isometry. %and $L^2(S_t)$ is closed in $L^2(\HS_t)$ by completeness of $L^2$ spaces.
Observe that if $\io$ is the involution of $\HS_t$ given by the orientation double cover, then $L^2(S_t)$ is an $\io^*$ invariant subspace of $L^2(\HS_t)$. Hence it is closed since $\io^*$ acts continuously on $L^2(\HS_t)$.

On $\HS_t$, away from the singularities we can find unit vector fields $X,Y$ which are tangent to the horizontal, vertical foliations of $\al^2$ respectively. By a result of Forni, $L^2(\HS_t)$ decomposes as the orthogonal direct sum of meromorphic functions on $\HS_t = (\HS, g_t \dt \al)$ and functions in the image of the anti-Cauchy-Riemann operator on $\HS_t$ determined by $g_t \dt \al$:
\begin{equation}
\d_t = e^t X + ie^{-t} Y
\end{equation}
(see \cite[Proposition 3.2]{MR1477760}). Hence we can write $p^*u = \Hm_t + \d_t v_t$ in $L^2(S_t)$, where the decomposition is again unique. By our previous computations we must also have $p^*u = p^* m_t + p^* h_t$, with $p^*m_t$ meromorphic on $\HS_t$ and orthogonal to $p^*h_t$. Hence by uniqueness of the decomposition, we must have $\Hm_t = p^* m_t$ and $\d_t v_t = p^* h_t$.

\end{proof}

We assume without loss of generality that $v_t$ is a function of zero average, since this assumption does not affect $\d_t v_t$. By this assumption, and because $u$ is real valued, it follows  that $v_t$ is purely imaginary. Since $p^*u$ is  flow invariant, we have $X(p^*u) = 0$, so we also get the following fact about $m_t$ from \cite[Lemma 2]{MR3177386}.
\begin{equation}
\f{\Rd}{\Rd t} \|p^*m_t\|^2_{L^2(\HS)} = 4\|\Im(p^* m_t) \|^2_{L^2(\HS)}
\label{eq:mero ODE}
\end{equation}
We will use this differential equation and the decomposition of the surface $S_t$ to show that $\|m_t\|_{L^2(S)} = 0$.

Fix $\et > 0$, and find the resulting $\ep > 0$ and the surface $S_{\ep,t}$. For a fixed $t > 0$ we can decompose $m_t = R_t + iI_t$ into the real and imaginary parts. For a fixed $z \in p^{-1}\b(S_{\ep,t}\e)$, and an $\al_t$ disc of radius $R < \ep$ centered at $z$ in a small enough flat chart, we can use the Cauchy integral formula for derivatives to compute:
\[
X_t \b(p^*I_t(z) \e) = \f{1}{\pi} \int_0^{2\pi} \frac{\cos(s)}{R} p^* I_t(Re^{is}) \, \Rd s
\]
We can then integrate the above equation over a small annulus near $z$ to obtain the bound:
\[
\b| X_t \b(p^*I_t (z) \e) \e| \le \f{4}{\ep^2} \|p^*I_t\|_{L^2(\HS)}
\]
The same bound holds for $Y_t \b(p^* I_t \e)$ by the same computation, which gives us the trivial bound:
\[
\b\|\nabla_t (p^* I_t)\e\|_{L^{\oo}(p^{-1}(S_{\ep,t}))} \le \f{8}{\ep^2} \|p^*I_t\|_{L^2(\HS)}
\]
Here $\nabla_t$ is the gradient given by the flat metric $\al_t$. The same bound applies to $p^*R_t$ by the Cauchy-Riemann equations.

Now suppose $z \in S \m (S_{\ep,t} \u \Si)$. If $0 < \rho < \ep$ is the distance of $z$ to $\Si$, then similarly:
\begin{equation}
\b|\nabla_t \b( p^*R_t(z)\e)\e| \le \frac{8}{\rho^2} \|p^* I_t\|_{L^2(\HS)}
\label{eq:distance squared}
\end{equation}
The same bound holds for $p^*I_t$.

Since the branched covering map $p$ is a local isometry outside of a finite set and $\|p^*f\|_{L^2(\HS_t)} = \|f\|_{L^2(S_t)}$, we see that these gradient bounds also apply to $I_t$ and $R_t$. We will use this fact in the next lemma to bound the norm of $m_t$ using the geometry of $S_t$. Morally, the bound is a constant divided by the injectivity radius of $S_t$.

%I would like to make the next lemma work in the non compact case, but alas.

\begin{lem}
There is a constant $C > 0$ depending only on the stratum of the quadratic differential $\Ph$ so that for any points $a,b$ in $S_{\ep,t}$ the following bound holds.
\begin{equation}
\b|R_t(a) - R_t(b)\e| \le C \b(\f{1}{\ep^2} \sum_{i = 1}^{C_t} \CD^i_t + \f{1}{\de_t} \e) \|I_t\|_{L^2(S)}
\end{equation}
%\[
%|R_t(a) - R_t(b)| \le\f{C}{\ka(t)} \|I_t\|_{L^2(S)}
%\]
The same bound holds for $\b|I_t(a) - I_t(b)\e|$.
\label{lem:essential gradient bound}
\end{lem}

\begin{proof} First we recall that the parameters in the bound were defined in the remarks before the proof of the theorem. We will recall the definitions as needed in the proof of this lemma.

If $\Ph$ is holomorphic, let $\ga$ be the unique flat geodesic on $S_t$ connecting $a$ and $b$ with minimal length. If $\Ph$ is meromorphic, then each of the at most simple poles of $\Ph$ lifts to a regular point of $\al$, so that $\al$ is a holomorphic abelian differential on $\HS$. In this case, let $\be$ be a flat geodesic on $\HS_t$ which is the shortest among all geodesics connecting one of the two points of $p\1(a)$ to one of the two points of $p\1(b)$, which is possibly not unique. Let $\ga = p(\be)$, then $\ga$ is a path from $a$ to $b$ on $S$. Either way, the curve $\ga$ is a union of straight line paths between singularities or between an endpoint and a singularity. For existence of these minimal length paths, see \cite[Chapter V, Section 18]{MR743423}.

In the meromorphic case, the following bounds may need a factor of 2 coming from the covering map being degree 2: $\ga$ might pass through the same subsurface of the decomposition multiple times, but this will happen at most twice (as otherwise we could have chosen a shorter path $\be$ on $\HS$). This does not change the result of the lemma, and in the rest of the proof we will assume that $\Ph$ is holomorphic.

Based on our gradient computations, we arrive at the following bound:
\begin{equation}
\b|R_t(z_i) - R_t(z_j)\e| \le 8\|I_t\|_{L^2(S)} \int_0^1  \b| \f{\Rd s}{\b(\mathrm{dist}_t(\ga(s),\Si)\e)^2} \e|
\label{eq:essential integral}
\end{equation}
We are interested in bounding the integral. There are three parts of the curve $\ga$ we need to consider. First is the part of $\ga$ in the "thick" part of the surface: let $Thick$ be the set of $s$ where $\mathrm{dist}_t(\ga(s),\Si) \ge \ep$. Here we do a trivial bound: the length of $\ga$ inside $Thick$, since it was chosen to have minimal length, is less than the sum of the lengths $\CD^i_t$ of diameters of each of the subsurfaces $S_t^i$ of the thick part. The maximum value of the function we are integrating is $\ep^{-2}$. So we have the bound:
\begin{equation}
\int_{Thick}\b| \f{\Rd s}{\b(\mathrm{dist}_t(\ga(s),\Si)\e)^2} \e| \le \b(\f{1}{\ep^2} \sum_{i = 1}^{C_t} \CD^i_t \e)
\label{eq:thick bound}
\end{equation}
Now we want to bound the integral over the thin part. Here the geodesic might pass through singularities $\si \in \Si$, so we want to modify our curve $\ga$ to avoid the singularities. By the definition of $\de_t$, we know that after removing the open ball of radius $\de_t$ about each $\si \in \Si$, the resulting surface is still path connected. Call this new surface $S_t \m D$. Then $\ga \i (S_t \m D)$ consists of at most $|\Si| + 1$ connected components, since we removed at most $|\Si|$ disjoint connected open sets from $S_t$. By choice of $\ga$, this holds for the domain of the parametrized curve as well as its image.

\begin{figure}[h]	% The [h] means "insert here." To really force the figure to be where you want it use [H]. The indenting shown below is only for readability.
	\centering
	\includegraphics[scale = 0.5]{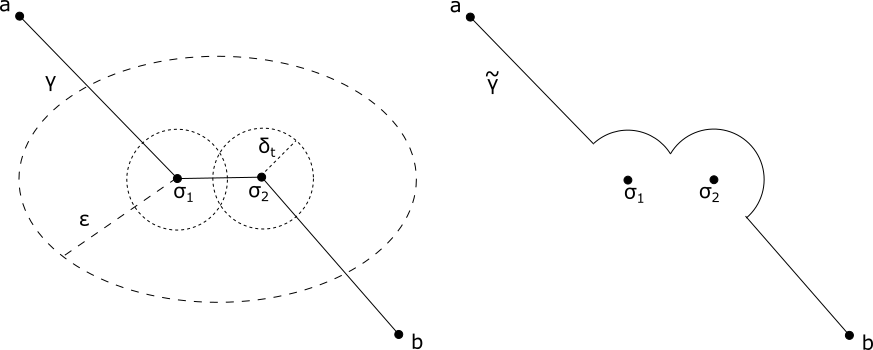}	% You can scale (here, by 70%). You could also set the width by using [width=3.0in].																			%   If you don't use any [] argument, the figure will be inserted at its original size. ex [scale = 0.8]
	\caption{Left: the geodesic $\ga$ connecting $a$ to $b$, which passes through the "thin" part determined by $\ep$. Right: the modified curve $\T\ga$, which stays at a distance at least $\de_t$ from all singularities $\si$.}
	\label{fig:curve mod}
\end{figure}

To connect each of these components of $\ga$, we can travel along the boundary of $S_t \m D$ at a distance $\de_t$ from $\Si$. The worst case scenario for connecting two components of $\ga \i (S_t \m D)$ is travelling around the circumference of all of the open balls surrounding every singularity to reach the next component of $\ga \i (S_t \m D)$, which must start at a distance $\de_t$ from a singularity. If $2\pi c > 0$ is the maximum angle around a conical singularity $\si$, then for each connection between components the worst case scenario is travelling a distance $|\Si| \times 2 \pi c \de_t$. We have to do this process at most $|\Si|$ times, since there are at most $|\Si| + 1$ components needing to be connected. In this way we obtain a new curve connecting $a$ and $b$, which is a union of geodesic segments and circular arcs. We will call this new curve $\Tga$. See Figure \ref{fig:curve mod} for a schematic of how to obtain $\T\ga$ from $\ga$.

Let $Thin_1$ be the set of $s$ where $\T\ga(s)$ is a distance $\de_t$ from the singularity. Since all of the previous work is for $Thin_1$, we use the trivial bound to obtain an upper bound on the first thin part:
\begin{equation}
\int_{Thin_1} \b| \f{\Rd s}{\b(\mathrm{dist}_t(\ga(s),\Si)\e)^2} \e| \le  \b(\f{2 \pi c |\Si|^2 \de_t }{\de_t^2}\e) =  \b(\f{2 \pi c |\Si|^2 }{\de_t}\e) 
\label{eq:thin bound 1}
\end{equation}

The third part of the integral is the set $Thin_2$ of $s$ so that $\de_t < \mathrm{dist}_t(\ga(s),\Si) < \ep$, which is still within the thin part but not controlled by our other estimate. %First, we split up the distance from the singularities with a trivial bound.
%\[
%\f{1}{\b(\mathrm{dist}_t(\T\ga(s),\Si)\e)^2} \le \sum_{\si \in \Si} \f{1}{\b(\mathrm{dist}_t(\T\ga(s),\si)\e)^2}
%\]
For a fixed singularity $\si$ and $k \ge 0$, consider the open annuli $A_k(\si) = \mathrm{ann}(\si,2^k\de_t, 2^{k+1} \de_t)$. For small enough values of $k$, the intersection of $\T\ga$ with $A_k(\si)$, excluding pieces of $\T\ga$ already bounded in $Thin_1$, is a union of line segments along rays starting from $\si$. Hence their intersection with the annulus has length less than the diameter of the annulus, $2 \times 2^{k+1}\de_t$. Since the distance to the singularity is at most $2^k \de_t$, the contribution of $A_k(\si)$ to the integral using the trivial bound is at most:
\[
\f{2^{k + 2}\de_t}{\b(2^k \de_t \e)^2} = \f{4}{2^k \de_t}
\]
The intersection of  $\T\ga$ and $Thin_2$ can be covered by such annuli in the following sense: the curve is a disjoint union of line segments, and each line segment is included in one of these annuli as a segment of a ray starting from a singularity. So for one singularity, we can bound the contribution to the integral by the sum of the previous bound over all $k$:
%\[
%\int_{Thin_2} \f{|\mathrm{d}s|}{\b(\mathrm{dist}_t(\T\ga(s),\si)\e)^2} \le \sum_{k = 0}^{\oo} \f{4}{2^k \de_t} = \f{8}{\de_t}
%\]
\[
\sum_{k = 0}^{\oo} \f{4}{2^k \de_t} = \f{8}{\de_t}
\]

\begin{figure}[h]	% The [h] means "insert here." To really force the figure to be where you want it use [H]. The indenting shown below is only for readability.
	\centering
	\includegraphics[scale = 0.5]{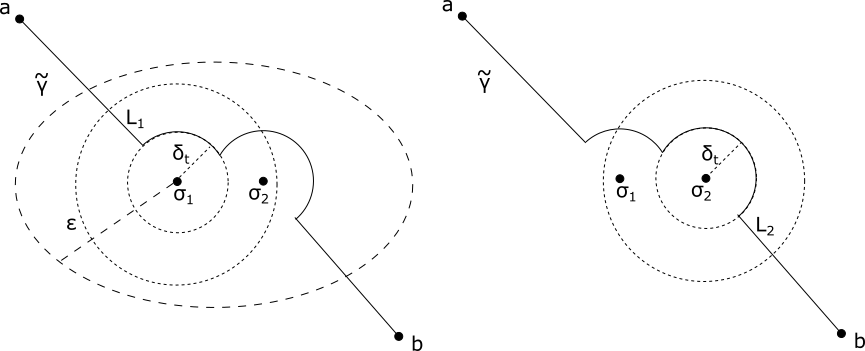}	% You can scale (here, by 70%). You could also set the width by using [width=3.0in].																			%   If you don't use any [] argument, the figure will be inserted at its original size. ex [scale = 0.8]
	\caption{Left: the annulus $A_0(\si_1)$, which contains the line segment $L_1$ of $\T\ga$ along a line starting from $\si_1$. The annulus $A_0(\si_1)$ controls the integral along $L_1$, but the annulus $A_1(\si_1)$ is needed to control the integral over the rest of that line in $Thin_2$. Right: the annulus $A_0(\si_2)$, which controls the integral over the line segment $L_2$.}
	\label{fig:annulus bound}
\end{figure}

Summing over the contributions from all singularities, we obtain the following bound on the third and final part of the integral:
\begin{equation}
\int_{Thin_2}\b| \f{\Rd s}{\b(\mathrm{dist}_t(\ga(s),\Si)\e)^2} \e| \le  \b( \f{8|\Si|}{\de_t}\e)
\label{eq:thin bound 2}
\end{equation}

Combining (\ref{eq:thick bound}), (\ref{eq:thin bound 1}), and (\ref{eq:thin bound 2}), we obtain the desired bound on (\ref{eq:essential integral}).
\begin{align*}
\b|R_t(z_i) - R_t(z_j)\e| &\le 8\|I_t\|_{L^2(S)} \b(\f{1}{\ep^2} \sum_{i = 1}^{C_t} \CD^i_t + \f{2 \pi c |\Si|^2 }{\de_t} + \f{8|\Si|}{\de_t}\e)\\
&\le C \b(\f{1}{\ep^2} \sum_{i = 1}^{C_t} \CD^i_t + \f{1}{\de_t}\e)\|I_t\|_{L^2(S)}
\end{align*}
Since $C$ depends only on the number of singularities and the maximum cone angle of a singularity, it depends only on the stratum of $\Ph$.

\end{proof}

\begin{lem}
The meromorphic functions $m_t$ are zero.
\label{lem:mero zero}
\end{lem}

\begin{proof}
Consider $p^*m_t = p^* R_t + i p^* I_t$. If $\ze_1, \ze_2 \in p\1(S_{\ep,t})$, then we observe each $p(\ze_k)$ is some $z_k \in S_{\ep,t}$, so that by definition of $p^*R_t$ and Lemma \ref{lem:essential gradient bound}:
\[
\b|p^* R_t(\ze_1) - p^*R_t (\ze_2) \e| \le C \b(\f{1}{\ep^2} \sum_{i = 1}^{C_t} \CD^i_t + \f{1}{\de_t} \e) \|I_t\|_{L^2(S)}
\]
The same bound holds for $p^* I_t$. Also observe that the Lebesgue measure of $p\1(S_{\ep,t})$ is at least $1 - \et$ (as Lebesgue downstairs is the pushforward by $p$ of Lebesgue upstairs).

Observe that we must have:
\[
\liminf_{t \to \oo} \b(\f{1}{\ep^2} \sum_{i = 1}^{C_t} \CD^i_t + \f{1}{\de_t} \e) \|p^*I_t\|_{L^2(\HS)} = 0
\]
If not, then as in \cite[Proof of Theorem 2]{MR3177386} there would be some constant $c$ so that for all (large enough) $t$:
\[
0 < c < \b(\f{1}{\ep^2} \sum_{i = 1}^{C_t} \CD^i_t + \f{1}{\de_t} \e) \|p^*I_t\|_{L^2(\HS)}
\]
Rearrange, square both sides, and integrate over the Teichm\"uller ray:
\begin{equation}
c^2 \int_0^\oo \b(\f{1}{\ep^2} \sum_{i = 1}^{C_t} \CD^i_t + \f{1}{\de_t} \e)^{-2} \, \Rd t < \int_0^\oo \|p^*I_t\|_{L^2(\HS)}^2 \, \Rd t \le \|p^* u\|_{L^2(\HS)} < \oo
\label{eq:finite integral}
\end{equation}
The second inequality follows from the differential equation (\ref{eq:mero ODE}) by integrating the imaginary part of $p^* m_t$. However, by equation (\ref{eq:systole2}), we can show that
\begin{equation}
\int_0^\oo \b(\f{1}{\ep^2} \sum_{i = 1}^{C_t} \CD^i_t + \f{1}{\de_t} \e)^{-2} \, \Rd t= \oo
\label{eq:infinite integral}
\end{equation}
which is a contradiction.

Indeed, first note that the number of components of the decomposition is uniformly bounded in time by a constant which depends only on the strata of $\Ph$. By \cite[Corollary 5.6]{MR1135877}, there is a constant $K$ so that if $D_t$ is the diameter of $S_t$ and $\ka(t)$ is the systole, and $D_t > \sqrt{2/\pi}$, then:
\begin{equation}
D_t \le \f{K}{\ka(t)}
\label{eq:systole diameter}
\end{equation}
Because the systole $\ka(t)$ is $2\de_t$ and $\ep$ is constant, we can see by a computation that divergence of the integral in (\ref{eq:systole2}) implies divergence of the integral in  (\ref{eq:infinite integral}), contradicting our hypothesis (\ref{eq:finite integral}).

Hence there must be arbitrarily large values of $t$ so that:
\[
\|p^*m_t\|_{L^\oo(p\1( S_{\ep,t}))} < \et
\]
We can use this to bound the $L^2$ norm of $m_t$. Let $\tau$ be one of these values of $t$. Then we can compute:
\[
\|p^* m_\tau\|_{L^2(\HS)}^2 = \int_{p\1(S_{\ep,\tau})} p^* m_\tau p^* u \, \Rd \la + \int_{\HS \m p\1(S_{\ep,\tau})} p^* m_\tau p^* u \, \Rd \la
\]
(note that because $p^*u$ is written as an orthogonal sum $p^*u = \d_\tau v_\tau + p^*m_\tau$, any terms involving $\d_\tau v$ vanish in the inner product). The idea of the bound is the following: the first integral is over a good set where $p^*m$ is small, and the second integral is over a bad set of small measure. More precisely:
\begin{align*}
\|p^* m_\tau\|_{L^2(\HS)}^2 &\le \|p^*m_t\|_{L^\oo(p\1( S_{\ep,\tau}))} \|p^*u\|_{L^1(\HS)} + \|p^*u\|_{L^\oo(\HS)} \int_{\HS \m p\1(S_{\ep,\tau})} |p^* m_\tau| \, \Rd \la \\
 &\le \et \|p^*u\|_{L^1(\HS)} + \|p^*u\|_{L^\oo(\HS)} (\la(p\1(S \m  S_{\ep,\tau}))^{1/2} \|p^* m_\tau\|_{L^2(\HS)}\\
 &\le \sqrt{\et} \b(\sqrt{\et} + \|p^*u\|_{L^\oo(\HS)} \e) \| p^*u \|_{L^2(\HS)} 
\end{align*}

Note that if $\|u\|_{L^{\oo}} = \oo$, we can replace it with an $L^2$ function of bounded norm without loss of generality, since level sets of invariant functions are invariant sets. Hence we can bound $\|p^*m_\tau\|_{L^2(\HS)} < \et$ for arbitrarily large values of $\tau$. Since $\et$ was arbitrary, by the differential equation (\ref{eq:mero ODE}), we see that $p^*m_t$ and hence $m_t$ must be zero for all $t$.

\end{proof}

Since $m_t = 0$, it suffices to show that $h_t = 0$ to conclude $u = 0$ almost everywhere. Following the argument in \cite[Proof of Theorem 2]{MR3177386}, it suffices to show that the horizontal foliation on $S_t$ is aperiodic to conclude that $h_t = 0$. Observe that there are no horizontal cylinders in $(S,\Ph)$. Indeed, if there was a horizontal cylinder, then $\de_t$ would be forced to decay exponentially, contradicting divergence of the integral in equation (\ref{eq:systole2}). A horizontal cylinder in $\HS$ would project to one on $S$, so there are no cylinders upstairs either. Hence, the horizontal foliation is aperiodic.

%Secondly, consider the set $B \se S_t$ of points in divergent leaves, and $\HB \se \HS_t$ of points in divergent flow trajectories. Observe that $\HB$ has measure zero. Indeed, let $\{K_n\}_{n = 1}^{\oo}$ be a compact exhaustion of $S_t$. Since every point on $S_t$ has either one or two pre-images under $p$ and $p$ is continuous, $\{p^{-1}(K_n)\}_{n = 1}^{\oo}$ is a compact exhaustion of $\HS_t$. So if a flow trajectory $L$ leaves every compact set of $\HS_t$, it must leave each of the $\{p^{-1}(K_n)\}_{n = 1}^{\oo}$. Then $p(L)$ must leave every compact set $K_n$, and so every compact set of $S_t$ since the $K_n$ form a compact exhaustion. So $\HB \se p^{-1}(B)$. Since Lebesgue measure on $S_t$ is the push-forward by $p$ of Lebesgue measure on $\HS_t$, and $B$ has measure zero, it follows that $\HB$ also has measure zero.

Now suppose that $a$ is a recurrent point. Let $\be_a$ be a short vertical segment through $a$, contained in a neighborhood which lifts homeomorphically by $p$ to $\HS_t$. By recurrence, there exists a sequence of points $a_k$ in the intersection of the horizontal leaf through $a$ and $\be_a$ such that $a_k \to a$, and $a_k \ne a$ by aperiodicity. Identify these points with their lifts in $\HS_t$. Since $v_t$ is continuous and $X(v_t) = 0$, we see that $v_t(a_k) = v_t(a)$ for all $k$. Since $p^*h_t = iY(v_t)$ and each of the $a_k$ lies along a vertical leaf through $a$, we can calculate that $p^*h_t(a) = Y(v_t)(a) = 0$. Since the set of recurrent points has full measure by Poincar\'e recurrence, we see that $h_t = 0$ Lebesgue almost everywhere. As our invariant function $u$ is equal to $m_t + h_t$, we conclude that $u = 0$ almost everywhere.

%As there are no horizontal cylinders and $S$ is compact, the horizontal foliation on $S_t$ is minimal, and hence recurrent. 
So if $u \in L^2(S)$  is constant along leaves, then $u$ is almost everywhere constant. We conclude that Lebesgue measure is ergodic for the horizontal foliation of $\Ph$.

\end{proof}

\B{Technical Remark:} Using only the bounds coming from the Cauchy integral formula in equation \ref{eq:distance squared}, we can get a bound of size roughly $C\de_t^{-3}\|I_t\|$ on the difference between values of $m_t$ at any two points in the thick part. This is because the length of the geodesic $\ga$ travelling through in the thin part is bounded by $\et \de_t^{-1} \le \de_t^{-1}$, and because the maximum value of the function being integrated is $\de_t^{-2}$. The length bound follows because we can embed a disc of radius $\de_t$ at any point on the geodesic. Since the measure of that set must be less than $\et$, we have $\l(\ga)\de_t \le \et$.

We can improve this bound to to $C\de_t^{-2}\|I_t\|$ because the integral is an average. Since the length of the curve is bounded, using a standard argument the value of the integral can be estimated by:
\[
\f{\et}{\de_t} \int_{\de_t}^{\oo} \f{1}{x^2} \, \Rd x = \f{\et}{\de_t} \times \f{1}{\de_t} \le \f{1}{\de_t^2}
\]
Following the rest of the proof, this would give ergodicity when the integral of $\ka(t)^4$ over a Teichm\"uller ray diverges. The purpose of Lemma \ref{lem:essential gradient bound} is to use flat geometry and compactness to improve this bound to $C\de_t^{-1}$, which improves the criterion to the integral of $\ka(t)^2$. We do not believe that the exponent can be reduced any further.

So in Trevi\aTL{n}o's setting for non-compact surfaces we have a bound of the form:
\begin{equation}
\b|m_t(\ze_1) - m_t (\ze_2) \e| \le C \b(\f{1}{\ep^2} \sum_{i = 1}^{C_t} \CD^i_t + \f{1}{\de_t^2} \e) \|I_t\|_{L^2(S)}
\label{eq:noncompact bound}
\end{equation}
and so the integral condition which we can prove guarantees ergodicity is:
\begin{equation}
\int_0^\oo \b(\f{1}{\ep^2} \sum_{i = 1}^{C_t} \CD^i_t + \f{1}{\de_t^2} \e)^{-2} \, \Rd t= \oo
\label{eq:infinite integral noncompact}
\end{equation}
For non-compact surfaces it may be possible to use flat geometry to lower the exponent on $\de_t$ in the integral just as in the compact case, but the geometry is harder to analyze.

\section{A Unique Ergodicity Criterion}
\label{sec:foliation unique ergodicity}

\begin{proof}[Proof of Theorem \ref{th:foliation unique ergodicity}]
First we apply Theorem \ref{th:foliation ergodicity},  and hence the horizontal foliation given by $\Ph$ is ergodic with respect to Lebesgue measure.

%Let $\ep(t) = \ep_0$ for some sufficiently small constant $\ep_0 > 0$, and define the surface decomposition by $S_{\ep(t),t} = \{z \in S : \mathrm{dist}_t(z,\Si) \ge \ep_0\}$.

%In addition, the quantity $\de_t$ of Theorem \ref{th:foliation ergodicity} is at least $\ka(t)/2$. Indeed, $\de_t$ is the flat injectivity radius of the surface $S_t$, which is half of the length of the shortest essential simple closed curve.
%Indeed, the worst case scenario is if $\ka(t)$ is realized by the core circle of a cylinder passing through all singularities, evenly spaced around the curve. In this case, passing halfway between two singularities gives $\de_t = \ka(t)/(2|\Si|)$. (but injectivity radius argument is better)
%By a computation, the above implies that the decomposition satisfies the hypothesis of Theorem \ref{th:foliation ergodicity},

To upgrade to unique ergodicity, we use an argument of Veech \cite[Section 1]{MR516048}. Suppose that there is another invariant measure $\mu$ which is ergodic for the horizontal foliation of $\Ph$. Then $\mu$ is non-atomic, since the foliation must be minimal by the systole condition, and must be mutually singular with respect to the ergodic Lebesgue measure $\la_{\Ph}$. For any fixed $s \in (0,1)$, the measure $\mu(s) = s\la_{\Ph} + (1-s)\mu$ is invariant for the horizontal foliation. If we can show that $\mu(s)$ is ergodic, then we have a contradiction because it is not extreme in the simplex of invariant measures.

On the double cover $(\HS,\al)$, there is the Lebesgue measure $\la_\al$ and a measure $\H{\mu}$ which is flow invariant such that $p_*\H{\mu} = \mu$. We can also request that $\io_* \H{\mu} = \H{\mu}$: if $\nu$ is flow invariant and $p_* \nu = \mu$, then consider the measure $\io_* \nu$. By definition of the projection map $p$, we must have $p_* \io_*(\nu) = \mu$. As the horizontal flow commutes with the involution $\io$, $\io_* \nu$ is also flow invariant. So $\f{1}{2} \b(\nu + \io_* \nu \e)$ is a flow invariant measure which projects to $\mu$ and is fixed by the involution.

Define $\H{\mu}(s) = s\la_\al + (1 - s) \H{\mu}$, so that $p_* \H{\mu}(s) = \mu(s)$. Then by \cite[Proof of Theorem 3]{MR3177386} there exists a homeomorphism $\HF_s$ from $\HS$ to itself such that:
\begin{enumerate}
\item $\HF_s$ induces a new translation structure on $\HS$.
\item $\HF_s$ preserves the horizontal foliation on $\HS$ and fixes the singularities.
\item The new translation structure is induced by a unique abelian differential $\al(s)$ with the same horizontal foliation as $\al$.
\item $\H{\mu}(s)$ pushes forward to the Lebesgue measure of $\al(s)$:

\begin{equation}
(\HF_s)_* \H{\mu}(s) = \la_{\al(s)}
\end{equation}

\item If $\ga$ is a homotopically nontrivial simple closed curve or a saddle connection on $\HS$, then $\HF_s(\ga)$ satisfies:

\begin{equation}
\mathrm{length}_{\al(s)}(\HF_s(\ga)) \ge s \dt \mathrm{length}_{\al}(\ga)
\label{eq:length change}
\end{equation}

\end{enumerate}
Define an involution $\si_s$ on $(\HS,\al(s))$ by $\si_s = \HF_s \o \io \o \HF_s\1$. Modding out $(\HS,\al(s))$ by the action of $\si_s$ produces an orbifold $S'$ with a branched cover $p' : \HS \to S'$. The map $\HF_s$ descends to a homeomorphism $F_s : S \to S'$ such that $(F_s)_* \mu(s) = p'_* \la_{\al(s)}$. In fact, $\si_s$ is an isometry of $(\HS,\al(s))$ sending $\al(s)$ to $-\al(s)$, and so sends leaves of the horizontal flow to leaves. To show this, we recall that $\HF_s$ is defined on $\HS$ away from singularities of $\al$ in local coordinates centered at $0$ as:
\begin{equation}
\HF_s(x + iy) = x + \mathrm{sign}(y) \dt i \Upsilon_s(\l_y)
\end{equation}
where $\l_y$ is the line segment between $0$ and $y$. Here $\Upsilon_s$ is a transverse invariant measure defined by:
\begin{equation}
\Upsilon_s(\ga) = \lim_{\ep \to 0} \f{1}{\ep} \H\mu(s)([0,\ep] \dt \ga)
\end{equation}
where $[0,\ep] \dt \ga$ denotes the horizontal flow of the curve $\ga$ for all times between $0$ and $\ep$. Now, since $\H\mu(s)$ was chosen to be involution invariant, we see that 
$\Upsilon_s(\ga) = \Upsilon_s(\io (\ga))$. A computation in local coordinates then shows that $\si_s$ is an isometry of $(\HS,\al(s))$.

Hence $S'$ has a Riemann surface structure, and in fact a flat surface structure from a quadratic differential $\Ph(s)$ such that $(p')^*(\Ph(s)) = \al(s)^2$. Observe also that the homeomorphism $F_s$ takes leaves of the horizontal foliation on $(S,\Ph)$ to leaves on $(S',\Ph(s))$ because $\HF_s$ takes leaves to leaves. By our previous observations we have $(F_s)_* \mu(s) = \la_{\Ph(s)}$, so that $F_s$ takes our convex combination to Lebesgue measure.

Now consider the systoles on $S_t$ and $S'_t$. We claim their lengths are related by:
\begin{equation}
\ka_{S'}(t) \ge s \ka_S(t)
\label{eq:systole bound}
\end{equation}
Indeed, let $\ga$ be any simple closed nontrivial curve on $S_t$, which corresponds uniquely to a curve $F_s(\ga)$ on $S_t'$. Up to removing the singularities on $S_t$ and $\HS_t$, the map $p$ is a covering map, so we can lift $\ga$ to a finite union of curves $\be_i$ on $\HS_t$. See Figure \ref{fig:curve lengths} for a schematic of the lifting. By (\ref{eq:length change}), $\mathrm{length}_{g_t \dt \al(s)}(\HF_s(\be_i)) \ge s \dt \mathrm{length}_{g_t  \dt \al}(\be_i)$. By uniqueness of lifts of paths, the union of the $\HF_s(\be_i)$ is a lift of $F_s(\ga)$. Since the covering maps are isometries restricted to the $\be_i$, we compute:
\begin{equation*}
\mathrm{length}_{g_t \dt \Ph(s)}(F_s(\ga)) = \sum_i \mathrm{length}_{g_t \dt \al(s)}(\HF_s(\be_i)) \ge s \sum_i \mathrm{length}_{g_t \dt \al}(\be_i) = s \dt \mathrm{length}_{g_t \dt \Ph}(\ga)
\end{equation*}
which proves our claim.

\begin{figure}[h]	% The [h] means "insert here." To really force the figure to be where you want it use [H]. The indenting shown below is only for readability.
	\centering
	\includegraphics[scale = 0.5]{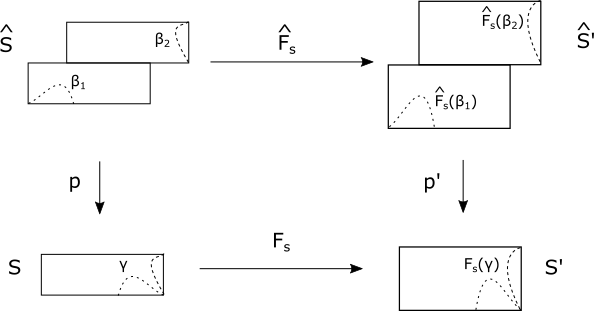}	% You can scale (here, by 70%). You could also set the width by using [width=3.0in].																			%   If you don't use any [] argument, the figure will be inserted at its original size. ex [scale = 0.8]
	\caption{The union of the curves $\be_1,\be_2$ projects to $\ga$. Since $\HF_s$ does not distort their lengths too much, $F_s$ does not distort the length of $\ga$ too much.}
	\label{fig:curve lengths}
\end{figure}

By (\ref{eq:systole bound}), Theorem \ref{th:foliation ergodicity} applies to $(S',\Ph(s))$:
\[
\int_0^{\oo} \ka_{S'}(t)^2 \, \Rd t\ge s^2 \int_0^{\oo} \ka_S(t)^2 \, \Rd t = \oo
\]
So the horizontal foliation of $(S',\Ph(s))$ is ergodic with respect to Lebesgue measure. But this is impossible, since by our construction Lebesgue measure $\la_{\Ph(s)}$ is a nontrivial convex combination of two invariant measures. Hence Lebesgue measure on $S$ must have been the only invariant measure to begin with, and so the horizontal foliation of $\Ph$ is uniquely ergodic.

\end{proof}

\section{Covering Spaces}
\label{sec:covers}

%If the quadratic differential $\Ph$ has no odd order zeroes or poles, then $p : \HS \to S$ is a genuine covering space. In this case there is a proof of unique ergodicity given Trevi\aTL{n}o's criterion which is much shorter than our proof in the general case.

Our proof of Theorem \ref{th:cover ergodicity} is a consequence of a proposition about lengths of curves. The necessary covering space theory can be found in \cite[Chapter 1]{MR1867354}.

\begin{prop}
Let $(M,\Ph)$ be a flat surface and  $P : \HM \to M$ a finite degree covering map. For every simple closed curve $\ga$ on $\HM$ which is not homotopically trivial, there is a simple closed nontrivial curve $\be$ on $M$ with $\mathrm{length}(\ga) \ge \mathrm{length}(\be)$.
\label{prop:covering space criterion}
\end{prop}

\begin{proof}
Observe that if the covering map has finite degree, then $\HM$ is also compact, and  $(\HM,P^* \Ph)$ is a compact flat surface. Let $\ga : [0,1] \to \HM$ be a piecewise smooth simple closed curve, and $P\ga$ its projection to $M$. Covering space theory tells us that $P_* : \pi_1(\HM) \to \pi_1(M)$ is an injective map. Hence if $[\ga] \ne 1 \in \pi_1(\HM)$, then $[P\ga] \ne 1 \in \pi_1(M)$. Since $P$ is a local isometry for the given flat structures, and the measurement of length is a local computation, we observe that, as parametrized curves, $\mathrm{length}(\ga) = \mathrm{length}(P\ga)$. We can view $\ga$ as a map $S^1 \to \HM$, and $\ga$ being simple tells us that we can assume this map is injective. If $P\ga : S^1 \to M$ is also injective, then $\mathrm{length}(\ga) = \mathrm{length}(P\ga)$ as curves, and so we can take $\be = P\ga$.

For the case when $P\ga$ is not injective, we find a subcurve which is injective. Since $P$ is a finite degree covering map and $\ga$ is injective, it follows that for each $t_0 \in [0,1)$, there are at most finitely many times $t_i \in [0,1)$ with $P\ga(t_0) = P\ga(t_i)$. So if $P\ga$ is not injective as a map from $S^1$, there is some minimal $s \in (0,1)$ such that $P\ga(s) = P\ga(r)$ for some $0 \le r < s < 1$. Define $\be$ to be the restriction of $P\ga$ to $[r,s]$. Minimality of $s$ means that $\be$ is injective as a map of $[r,s)$ into $M$, so the curve is simple, and closed because $P\ga(r) = P\ga(s)$ by construction. Since $\be$ is the restriction of $P\ga$ we see that $\mathrm{length}(\ga) = \mathrm{length}(P\ga) \ge \mathrm{length}(\be)$. Since $\ga$ and $\be$ are both injective maps, these lengths are their lengths as curves.

Finally, $\be$ is homotopically nontrivial. Indeed, consider the universal cover $\TM$ of $\HM$, and let $Q: \TM \to \HM$ be the covering map. Since $\HM$ is a cover of $M$, $\TM$ is also the universal cover of $M$. For any $m \in Q\1(\ga(0))$, there is a unique lift $\T\ga$ of $\ga$ with $\T\ga(0) = m$. Observe that $\T\ga$ is also a lift of $P\ga$. Since we assumed that $\ga$ was an injective map on $S^1$, $\T\ga : [0,1] \to \TM$ is also injective, where $\T\ga(0) \ne \T\ga(1)$ follows because $\ga$ is homotopically nontrivial. Since $\T\ga$ is injective, we see that $\T\ga(r) \ne \T\ga(s)$ in the universal cover but $P\ga(r) = P\ga(s)$. This proves that $\be$ is homotopically nontrivial.

\end{proof}

\begin{proof}[Proof of Theorem \ref{th:cover ergodicity}]

Proposition \ref{prop:covering space criterion} shows that the systoles on the covering space $R$ must satisfy:
\begin{equation}
\int_0^\oo \ka_{R}(t)^2 \, \Rd t \ge \int_0^\oo \ka_{S}(t)^2 \, \Rd t = \oo
\label{eq:cover comp}
\end{equation}
which by Theorem \ref{th:foliation unique ergodicity} implies that the horizontal foliation on $R$ is uniquely ergodic.

\end{proof}

If we only assume Trevi\aTL{n}o's Theorem \ref{th:rodrigo theorem 4}, then we have a short proof of Theorem \ref{th:foliation unique ergodicity} in the special case where the double covering map has no branching. Let $R$ in the Theorem be the orientation double cover $\HS$ of $S$. Since the systoles given by $(\HS,\al^2)$ are the same as the systoles given by $(\HS,\al)$, we see by equation (\ref{eq:cover comp}) that the horizontal flow on $(\HS,\al)$ is uniquely ergodic. By definition the simplex of invariant measures for $(S,\Ph)$ contains only the Lebesgue measure as an extreme point, so the horizontal foliation on $(S,\Ph)$ is uniquely ergodic.

The fact that covering maps induce an injective map on the fundamental group is necessary for the previous proof. Note that the in example of Theorem \ref{th:double NUE}, the induced map $p_*$ has a kernel on $\pi_1$. Indeed, let $\ga$ be a path formed by concatenating two paths $\ga_2 \dt \ga_1$, depicted in Figure \ref{fig:double kernel}.  Then $[\ga] \ne 1$ in $\pi_1(\HS,\si)$, where $\si$ is either singularity. However, $p_*([\ga]) = 1$, because $p(\ga_1)$ and $p(\ga_2)$ are the same curve but with opposite orientations, so the parametrized curve is homotopic to a constant path.

\begin{figure}[h]	% The [h] means "insert here." To really force the figure to be where you want it use [H]. The indenting shown below is only for readability.
	\centering
	\includegraphics[scale = 0.4]{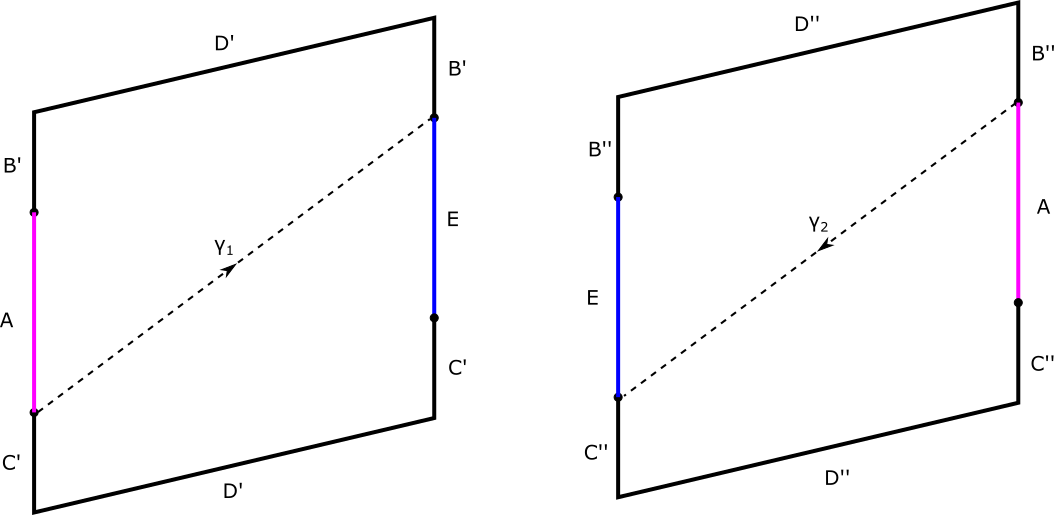}	% You can scale (here, by 70%). You could also set the width by using [width=3.0in].																			%   If you don't use any [] argument, the figure will be inserted at its original size. ex [scale = 0.8]
	\caption{A curve $\ga$ in the kernel of the induced map on $\pi_1$ in Theorem \ref{th:double NUE}. Here $\ga$ is the union of $\ga_1$, the path traveling on the upward diagonal line between the bottom left singularity and the top right singularity in the left polygon, and $\ga_2$, the path traveling on the same line in the reverse direction in the right polygon. Note that $\ga$ is separating, so is nontrivial.}
	\label{fig:double kernel}
\end{figure}

%\begin{cor}

%\begin{enumerate}
%\item If $(S,\al)$ is a translation surface satisfying Trevi\aTL{n}o's criterion, then given a finite degree covering map $p : R \to S$, the translation surface $(R,p^*\al)$ also satisfies Trevi%%\aTL{n}o's criterion (and hence the horizontal flow given by $p^*\al$ is uniquely ergodic).
%\item If $(S,\Ph)$ satisfies Trevi\aTL{n}o's criterion and its orientation double cover is a covering space, then the horizontal flow on its double cover $(\HS,\al)$ is uniquely ergodic (and hence the horizontal foliation given by $\Ph$ is uniquely ergodic).
%\end{enumerate}
%\label{cor:covering space UE}
%\end{cor}

%\phantomsection
%\addcontentsline{toc}{section}{Unproven Theorems}
%\renewcommand{\listtheoremname}{Unproven Results (proofs just require details)}
%\listoftheorems[numwidth=3em,ignoreall,show={thm0,prop0,lem0,cor0,exm0}]
%\renewcommand{\listtheoremname}{Unproven Results (proofs require more thought)}
%\listoftheorems[numwidth=3em,ignoreall,show={thm1,prop1,lem1,cor1,exm1}]
%\renewcommand{\listtheoremname}{Unproven Results (proofs require deeper ideas)}
%\listoftheorems[numwidth=3em,ignoreall,show={thm2,prop2,lem2,cor2,exm2}]
%\renewcommand{\listtheoremname}{Unproven Results (proofs beyond our scope)}
%\listoftheorems[numwidth=3em,ignoreall,show={thm3,prop3,lem3,cor3,exm3}]

\phantomsection
\addcontentsline{toc}{section}{References}
\bibliographystyle{amsalpha}

\bibliography{UEQD_bib}

\end{document}